\documentclass{amsart}

\usepackage{amsmath,amssymb,enumerate}
\usepackage{mathptmx}

\usepackage{microtype}
\usepackage{colonequals}

\usepackage[colorlinks]{hyperref}

\usepackage[nobysame]{amsrefs}

\usepackage[all]{xy}
\SelectTips{cm}{}

\newtheorem{theorem}[subsection]{Theorem}

\newtheorem{proposition}[subsection]{Proposition}
\newtheorem{lemma}[subsection]{Lemma}
\newtheorem{corollary}[subsection]{Corollary}

\theoremstyle{definition}

\theoremstyle{remark}
\newtheorem{remark}[subsection]{Remark}

\newcommand{\add}{\operatorname{add}}
\newcommand{\dbcat}{\mathsf{D}^{\mathsf b}}
\newcommand{\dsing}{\mathsf{D}_{\mathsf{sg}}}
\newcommand{\Ext}{\operatorname{Ext}}
\newcommand{\ges}{\geqslant}
\newcommand{\hh}{\operatorname{H}}
\newcommand{\perf}{\operatorname{perf}}
\newcommand{\lra}{\longrightarrow}

\newcommand{\fp}{\mathfrak{p}}
\newcommand{\fq}{\mathfrak{q}}
\newcommand{\mcA}{\mathcal{A}}

\newcommand{\mcT}{\mathcal{T}}
\newcommand{\mcS}{\mathcal{S}}

\newcommand{\Min}{\operatorname{Min}}
\newcommand{\rmod}{{\operatorname{mod}\,}}
\newcommand{\Reg}{{\operatorname{Reg}\,}}

\newcommand{\Spec}{{\operatorname{Spec}\,}}
\newcommand{\Supp}{{\operatorname{Supp}}}
\newcommand{\susp}{\mathsf{\Sigma}}
\newcommand{\syz}{{\Omega}\,}
\newcommand{\thick}{\operatorname{thick}}

\begin{document}

\title[Regular loci and generators]{Openness of the regular locus and \\ generators for module categories}

\author{Srikanth B. Iyengar}
\address{Department of Mathematics, University of Utah, Salt Lake City, UT 84112-0090, USA}
\email{iyengar@math.utah.edu}

\author{Ryo Takahashi}
\address{Graduate School of Mathematics, Nagoya University, Furocho, Chikusaku, Nagoya 464-8602, Japan}
\email{takahashi@math.nagoya-u.ac.jp}

\date{\today}


\begin{abstract}
This work clarifies the relationship between the openness of the regular locus of a commutative Noetherian ring $R$ and the existence of generators for the category of finitely generated $R$-modules, the corresponding bounded derived category, and for the singularity category of $R$.
\end{abstract}

\maketitle

\section{Introduction}
This note concerns the relationship between the openness of the regular locus of a commutative Noetherian ring $R$ and the existence of (classical) generators, in the sense of Bondal and Van den Bergh~\cite{Bondal/vandenBergh:2003}, of its bounded derived category, $\dbcat(\rmod R)$. The regular locus of $R$, denoted $\Reg R$, is the subset of $\Spec R$ consisting of prime ideals $\fp$ such that the local ring $R_{\fp}$ is regular, whereas a generator for the bounded derived category is an $R$-complex in $\dbcat(\rmod R)$ from which one can build every $R$-complex in $\dbcat(\rmod R)$ using finite direct sums, direct summands, and exact triangles. It is not hard to check that when such a generator exists, $\Reg R$ is an open subset of $\Spec R$; see  Lemma~\ref{lem:singcl}. We suspect that the converse does not hold, but are unable to find an example to validate this hunch. What we can prove is the following statement, proved in Section~\ref{sec:proof}.

\begin{theorem}
\label{thm:only}
The  conditions below are equivalent for a commutative Noetherian ring $R$.
\begin{enumerate}[\quad\rm(1)]
\item
$\Reg(R/\fp)$ contains a nonempty open subset for each $\fp\in\Spec R$.
\item
$\Reg(R/\fp)$ is open for each $\fp\in\Spec R$.
\item
The abelian category $\rmod (R/\fp)$ has a generator for each $\fp\in\Spec R$.
\item
The triangulated category $\dbcat(\rmod (R/\fp))$ has a generator for each $\fp\in\Spec R$.
\item
The triangulated category $\dsing(R/\fp)$ has a generator for each $\fp\in\Spec R$.
\end{enumerate}
When these hold, the abelian category $\rmod R$ and the triangulated categories $\dbcat(\rmod R)$, $\dsing(R)$, have generators.
\end{theorem}

Here $\dsing(R/\fp)$ is the singularity category of $R/\fp$ introduced by Buchweitz~\cite{Buchweitz:1987}, namely the Verdier quotient of $\dbcat(\rmod(R/\fp))$ by the full subcategory of perfect complexes; see \ref{ss:perfect} and \ref{ss:sing}.  One corollary of the theorem  is that the bounded derived category of any excellent ring has a generator. The result also gives another perspective on Nagata's criterion for the openness of the regular locus; see~Corollary~\ref{cor:nagata}.

\section{Singular loci and generators}
\label{sec:proof}
Throughout this work $R$ will be a commutative Noetherian ring with identity. We write $\rmod R$ for the (abelian) category of finitely generated $R$-module and $\dbcat(\rmod R)$ for its derived category, viewed as a triangulated category; see, for instance, \cite{Gelfand/Manin:2003}*{Chapter IV}. 

\subsection{Perfect complexes}
\label{ss:perfect}
An $R$-complex (that is to say, a complex of $R$-modules) is  \emph{perfect} if it is quasi-isomorphic to a bounded complex of finitely generated projective $R$-modules. 

\begin{remark}
\label{rem:truncation}
Given any $R$-complex $X$ in $\dbcat(\rmod R)$ and integer $s\ge \sup\{i\mid \hh_{i}(X)\ne0\}$, there is an exact triangle
\[
P \lra X \lra \susp^{s}M \lra
\]
in $\dbcat(\rmod R)$ where $M$ is a finitely generated $R$-module, and $P$ is a perfect complex with $P_{i}=0$ for $i\ge s$; moreover, one can choose $P$ such that $P_{n}=0$ for $n < \inf\{i\mid \hh_{i}(X)\ne0\}$. 

Indeed, replacing $X$ by a suitable projective resolution, one can assume that each $X_{i}$ is a finitely generated projective, and $X_{n}=0$ for $n<\inf \{i\mid \hh_{i}(X)\ne0\}$. The subcomplex $X_{<s}$ of $X$ is perfect and the quotient complex $X_{\ges s}\colonequals X/X_{<s}$ has homology only in degree $s$,  by the choice of $s$, hence it is quasi-isomorphic to $\susp^{s}\hh_{s}(X_{\ges s})$. The exact sequence of $R$-complexes
\[
0\lra X_{<s} \lra X \lra X_{\ges s}\lra 0
\]
induces an exact triangle in $\dbcat(\rmod R)$ with the desired properties.
\end{remark}

Let $\Spec R$ denote the spectrum of prime ideals of $R$, with the Zariski topology.  The support of an $R$-module $M$ is denoted $\Supp_{R}M$; when $M$ is finitely generated, this consists precisely of the prime ideals containing the annihilator ideal of $M$, and in particular a closed subset of $\Spec R$. In what follows, for any $R$-module $M$ and an integer $n\ges 0$, we write $\syz_R^nM$, or just $\syz^nM$ when the ring in question is clear, for an $n$-th syzygy of $M$.

For any $R$-complex $X$ we write $\perf_{R}X$ for the locus of prime ideals where $X$ is perfect:
\[
\perf_{R}X\colonequals \{\fp\in\Spec R\mid \text{the $R_{\fp}$-complex $X_{\fp}$ is perfect.}\}
\]
The observation below is well-known, at least for finitely generated modules. 

\begin{lemma}
\label{lem:ipd}
For any $X$ in $\dbcat(\rmod R)$, the subset $\perf_{R}X$ of\ $\Spec R$ is open.
\end{lemma}

\begin{proof}
Set $s\colonequals \sup\{i\mid \hh_{i}(X)\ne0\}$ and let $M$ be the finitely generated $R$-module in the exact triangle in Remark~\ref{rem:truncation}. Since localization is an exact functor, for each $\fp\in\Spec R$, one gets an exact triangle
\[
P_{\fp} \lra X_{\fp} \lra \susp^{s}M_{\fp} \lra 
\]
in the bounded derived category of $R_{\fp}$. Since the $R_{\fp}$-complex $P_{\fp}$ is perfect, it follows that $X_{\fp}$ is perfect if and only if $M_{\fp}$ is perfect. It thus suffices to verify the desired claim when $X$ is a finitely generated $R$-module. In this case, $X$ is prefect if and only if it is of finite projective dimension.  For any integer $n\ge 1$, one has $\Ext^{n}_{R}(M,\syz^{n}M)=0$ if and only if the projective dimension of $M$ is $\le n-1$; see, for example, \cite{Iyengar/Takahashi:2016}*{Lemma 2.14}.  It follows that 
\[
\Spec R \setminus \perf M=\bigcap_{n\ges 1}\Supp_{R}\Ext^{n}_{R}(M,\syz^nM).
\]
It remains to observe that since $M$ is finitely generated, so is each $\Ext^{n}_{R}(M,\syz^{n}M)$.
\end{proof}

\subsection{Singularity category}
\label{ss:sing}
The perfect complexes form a thick subcategory of $\dbcat(\rmod R)$. The corresponding Verdier quotient is the \emph{singularity category} of $R$, introduced by Buchweitz~\cite{Buchweitz:1987} under the name `the stabilized derived category of $R$'. It follows from Remark~\ref{rem:truncation} that any object in $\dsing(R)$ is isomorphic to a finitely generated $R$-module. By construction an $R$-complex is zero in $\dsing(R)$ if and only if it perfect.

The singularity category is a measure of singularity of $R$, in that $\dsing(R)\equiv 0$ if and only if $R$ is regular, by which we mean the local ring $R_{\fp}$ is regular for each $\fp$ in $\Spec R$. 

Indeed, when $\dsing(R)\equiv0$, each finitely generated $R$-module is perfect, that is to say, of finite projective dimension, and this implies that $R$ is regular. Conversely, when $R$ is regular, it follows from \cite{Bass/Murthy:1967}*{Lemma~4.5} that the projective dimension of each finitely generated $R$-module is finite; equivalently, that each complex in $\dbcat(\rmod R)$ is perfect; see Remark~\ref{rem:truncation}.

\subsection{Generators}
\label{ss:generators}
Let $\mcA$ be an abelian category, for example, $\rmod R$. A \emph{thick} subcategory of $\mcA$ is a full subcategory $\mcS$ that is closed under direct summands, and has the two-out-of-three property for exact sequences: Given any exact sequence
\[
0\lra X\lra Y\lra Z\lra 0
\]
in $\mcA$, if two of $X,Y$, and $Z$ are in $\mcS$, then so is the third. Given an object $G$ in $\mcA$, we write $\thick_{\mcA}(G)$, or just $\thick(G)$ when the underlying abelian category is clear, for the smallest thick subcategory of $\mcA$ containing $G$; see, for example, \cite{Iyengar/Takahashi:2016}*{\S4.3} for a constructive definition of this category. The object $G$ is a \emph{generator} for $\mcA$ if $\thick(G)=\mcA$.

One has an analogous notion of a thick subcategory of a triangulated category, and of a generator of a triangulated category; see \cite{Bondal/vandenBergh:2003}*{Section 2.1}, where these are called ``classical generators''.  It is easy to verify that any generator for $\rmod R$ is a generator for $\dbcat(\rmod R)$, and that any generator for the latter is a generator for $\dsing(R)$.

\begin{lemma}
\label{lem:modulo}
If $\rmod (R/\fp)$ has a generator for each $\fp\in\Min R$, then so does $\rmod R$. The analogues for the bounded derived category, and the singularity category, also hold.
\end{lemma}

\begin{proof}
There exists a filtration $0=I_0\subseteq I_1\subseteq\cdots\subseteq I_n=R$ by ideals such that for each $i$ one has $I_i/I_{i-1}\cong R/\fq_i$ with $\fq_i\in\Spec R$. For each $i$, fix a minimal prime $\fp_i$ of $R$ contained in $\fq_i$. Given an $R$-module $M$, the sequence $0=I_0M\subseteq I_1M\subseteq\cdots\subseteq I_nM=M$ of $R$-submodules gives rise to exact sequences
\[
0\lra I_{i-1}M \lra I_iM \lra M_i\lra 0
\]
of $R$-modules where $M_i$ is an $(R/\fp_i)$-module.

Assume that for  $1\leqslant i\leqslant n$ there exists an $(R/\fp_{i})$-module $G_i$ that generates $\rmod(R/\fp_{i})$. Set $G\colonequals \oplus_{i} G_{i}$.  Using the exact sequences above, an obvious inductive argument yields that $G$ is a generator for $\rmod R$.

A similar argument applies also in the case of the bounded derived category and the singularity category. 
\end{proof}

Next we call the definitions concerning the regular locus $\Reg R$ of the ring $R$. 

\subsection{The J-conditions}
Let $R$ be a commutative Noetherian ring. Following Matsumura~\cite{Matsumura:1980}*{\S32.B} we say that
\begin{enumerate}[\quad\rm(1)]
\item
$R$ is \emph{J-0} if $\Reg R$ contains a nonempty open subset of $\Spec R$.
\item
$R$ is \emph{J-1} if $\Reg R$ is open in $\Spec R$.
\item
$R$ is \emph{J-2} if it satisfies the following equivalent conditions.
\begin{enumerate}[\quad\rm(a)]
\item
Any finitely generated $R$-algebra is J-1.
\item
Any module-finite $R$-algebra is J-1.
\end{enumerate}
\end{enumerate}

We record a few remarks concerning these notions.

\begin{remark}
\label{rem:Ji}
It is easy to verify that a ring is J-0 if and only if $\Spec R\setminus \Reg R$ is not dense in $\Spec R$, if and only if there exists $\fp\in\Spec R$ and $f\in R\setminus\fp$ such that the ring $R_f$ is regular.

An artinian local ring that is not a field is not J-0 since the regular locus is empty. There is an example of a one-dimensional domain that is not J-0; see \cite{Hochster:1973}*{Example 1}. 

A regular ring is J-1, and so is a local isolated singularity of positive dimension. The J-1 condition does not imply J-0: the regular locus might be empty. However, every J-1 domain is J-0, for then the zero ideal belongs to $\Reg R$. On the other hand, there exists a $3$-dimensional local domain $R$ that is J-0 but not J-1; see \cite{Ferrand/Raynaud:1970}*{Proposition 3.5}.

When   $R$ is J-2, it is J-1. Excellent rings, Artinian rings, $1$-dimensional local rings and $1$-dimensional Nagata rings are J-2; see \cite{Matsumura:1980}*{Remark 32.B}.
\end{remark}

\begin{lemma}
\label{lem:singcl}
If an $R$-complex $G$  generates $\dsing(R)$, then $\Reg R=\perf G$ and $R$ is J-1.
\end{lemma}

\begin{proof}
When $G$ generates $\dsing(R)$,  the $R_{\fp}$-complex $G_{\fp}$ generates $\dsing(R_{\fp})$ for any $\fp\in \Spec R$; this is straightforward to verify.
This implies  $\dsing(R_{\fp})\equiv 0$ if and only if $G_{\fp}=0$ in $\dsing(R_{\fp})$, that is to say, if and only if the $R_{\fp}$-complex $G_{\fp}$ is perfect. Since $\dsing(R_{\fp})\equiv 0$ holds precisely when $R_{\fp}$ is regular,  the equality $\Reg R=\perf G$ follows. The latter is open, by Lemma \ref{lem:ipd}, and hence the ring $R$ is J-1.
\end{proof}

We are now prepared to prove Theorem~\ref{thm:only} stated in the Introduction.

\begin{proof}[Proof of Theorem~\ref{thm:only}]
For a start, (3) $\Rightarrow$ (4) $\Rightarrow$ (5) are clear, while  (5) $\Rightarrow$ (2) is by Lemma \ref{lem:singcl}, and  (2) $\Rightarrow$ (1) is clear, for $R/\fp$ is a domain; see Remark \ref{rem:Ji}.

(1) $\Rightarrow$ (3): Assume that (3) does not hold and let $\fp$ be a maximal element, with respect to inclusion, with the property that $\rmod R/\fp$ does not have a generator. Replace $R$ by $R/\fp$, we may then assume that  $\rmod R$ does not have a generator, but that $\rmod (R/\fq)$ does for each nonzero prime ideal $\fq$ in $R$.  Since $R$ is a domain, the J-0 condition entails the existence of a nonzero element $f\in R$ such that $R_f$ is regular; see Remark~\ref{rem:Ji}. If $R$ is regular, then $R$  would be a generator for $\rmod R$. We may thus assume that $f$ is not a unit. Since  $\rmod (R/\fq)$ has a generator for any prime ideal $\fq$ of $R$ containing $f$, Lemma \ref{lem:modulo} implies that  $\rmod(R/(f))$ has a generator, say $G$.

We claim that $R\oplus G$ generates $\rmod R$, which is a contradiction.

Indeed, let $M$ be a finitely generated $R$-module. Since the ring $R_{f}$ is regular, the projective dimension of the $R_{f}$-module $M_{f}$ is finite, say equal to $d$; this is again by \cite{Bass/Murthy:1967}*{Lemma~4.5}. For $n=\max\{1,d\}$ and  $N=\syz^n_{R}M$ one has
\[
\Ext_R^1(N,\syz N)_f\cong\Ext_{R_f}^{n+1}(M_f,(\syz N)_f)=0.
\]
Hence $f^a\cdot \Ext_R^1(N,\syz N)=0$ for an integer $a\ge 1$. It follows that there is an exact sequence
\[
0 \lra (0:_Nf^{a}) \lra N\oplus\syz_RN \lra \syz_R(N/f^{a}N) \lra 0\,,
\]
of $R$-modules; see, for example, \cite{Iyengar/Takahashi:2016}*{Remark 2.12}. Since $f^a$ is $N$-regular, for $n\ge 1$ and $R$ is a domain, one obtains that $N$ is isomorphic to a direct summand of $\syz_R(N/f^aN)$.

For each integer $i\geqslant1$ there is an exact sequence
\[
0 \lra f^{i-1}N/f^iN \lra N/f^iN \lra N/f^{i-1}N \lra 0\,.
\]
Since each $f^{i-1}N/f^iN$ is an $R/(f)$-module, it is generated by $G$ as an $R/(f)$-module, and hence also as an $R$-module. Using the exact sequences above, a standard induction on $i$ then yields that $N/f^{i}N$, and in particular $N/f^{a}N$, is generated by $G$ as an $R$-module. Therefore $\syz_R(N/f^aN)$ is generated by $R\oplus G$, as $R$-modules. Finally, since $N$ is an $n$th syzygy of $M$, it follows that $M$ itself is generated by $R\oplus G$, as claimed. 

This completes the proof of the claim of the equivalences of conditions (1)--(5).

The last assertion in the theorem is a consequence of Lemma \ref{lem:modulo}. 
\end{proof}

One consequence is Nagata's criterion for regularity; see \cite{Matsumura:1980}*{32.A}

\begin{corollary}
\label{cor:nagata}
If\,  $\Reg(R/\fp)$ contains a nonempty open subset for each $\fp$ in $\Spec R$, then the subset $\Reg R$  is open.
\end{corollary}

\begin{proof}
The hypothesis implies that $\dsing(R)$ has a generator, by  Theorem~\ref{thm:only}, and hence that $\Reg R$ is open, by Lemma  \ref{lem:singcl}.
\end{proof}

Theorem~\ref{thm:only} yields also a characterization of the J-2 property.

\begin{proposition}
\label{prp:J2}
The following conditions are equivalent.
\begin{enumerate}[\quad\rm(1)]
\item
The ring $R$ is J-2.
\item
Any module-finite $R$-algebra domain is J-0.
\item
For any module-finite $R$-algebra  (equivalently, and a domain) $A$, the abelian category $\rmod A$ has a generator.
\item
For any module-finite $R$-algebra  (equivalently, and a domain) $A$, the  triangulated category $\dbcat(\rmod A)$ has a generator.
\item
For any module-finite $R$-algebra  (equivalently, and a domain) $A$, the  triangulated category  $\dsing(A)$ has a generator.
\end{enumerate}
\end{proposition}

\begin{proof}
$(1)\Rightarrow(2)$: This implication follows from the fact that any J-1 domain is J-0.

$(2)\Rightarrow(3)$: Let $A$ be a module-finite $R$-algebra. For any prime ideal $\fp$ of $A$, the residue ring $A/\fp$ is a module-finite $R$-algebra domain, and hence is J-0 by hypothesis. Theorem~\ref{thm:only} then implies that $\rmod A$ has a generator.

$(3)\Rightarrow(4)\Rightarrow(5)$: These implications are obvious.

$(5)\Rightarrow(1)$: The hypothesis is that the singularity category of any module-finite $R$-algebra domain has a generator. Let $A$ be a module-finite $R$-algebra. For any prime ideal $\fp$ of $A$, the ring $A/\fp$ is a module-finite $R$-algebra and a domain, and hence  $\dsing(A/\fp)$ has a generator. Therefore $\dsing(A)$ has a generator, by Lemma~\ref{lem:modulo}, and then Proposition \ref{lem:singcl} implies that $A$ is J-1. This proves that $R$ is J-2.
\end{proof}

\begin{remark}
\label{rem:todo}
For application, it is often important that $\dbcat(\rmod R)$ and $\dsing(R)$ have strong generators. Roughly speaking, a strong generator of a triangulated category $\mcT$ is a object $G$  such that, for some  integer $d$, each object in $\mcT$ can be built out of $\add G$, the additive closure of $G$, using $d$ exact triangles; see \cite{Bondal/vandenBergh:2003}*{Section~2.2} or \cite{Iyengar/Takahashi:2016}*{Section~7} for a  precise description. The reason for caring about them is that when they exist, under certain reasonable hypothesis on $\mcT$, cohomological functors on $\mcT$ are representable; see~\cite{Bondal/vandenBergh:2003}*{Theorem~1.3}.

It would be also of interest to find an analogue of Theorem~\ref{thm:only} dealing with strong generators. The existence of a generator does not imply that of a strong generator. For example, when $R$ is regular, $\dbcat(\rmod R)$ has a generator (namely, $R$ itself), but it has a strong generator if and only if the Krull dimension of $R$ is finite. One can prove that, at least when $R$ is reduced, the existence of a strong generator implies that $\dim R$ is finite. We do not know whether, when $\dim R$ is finite, the existence of a generator implies that of a strong generator. 
\end{remark}

\section*{Acknowledgments}
The authors are grateful to Shiro Goto, Kazuhiko Kurano, and Jun-ichi Nishimura for valuable suggestions. The first author was partially supported by the National Science Foundation, under grant DMS-1700985. The second author was partially supported by JSPS Grant-in-Aid for Scientific Research 16K05098 and JSPS Fund for the Promotion of Joint International Research 16KK0099. Finally, the authors thank the referee for reading the paper carefully and giving helpful comments.

\end{document}